\documentclass[11pt]{amsart}
\usepackage[T1]{fontenc}
\usepackage{lmodern,amsmath,amssymb,mathtools}
\usepackage[margin=1in]{geometry}
\usepackage[expansion=false]{microtype}
\usepackage{xcolor,array,booktabs,tabularx}
\usepackage[colorlinks=true,linkcolor=blue!50!black,citecolor=blue!50!black,urlcolor=blue!50!black]{hyperref}
\newtheorem{theorem}{Theorem}[section]
\newtheorem{lemma}[theorem]{Lemma}
\newtheorem{proposition}[theorem]{Proposition}
\newtheorem{corollary}[theorem]{Corollary}
\theoremstyle{definition}
\newtheorem{definition}[theorem]{Definition}
\newtheorem{remark}[theorem]{Remark}
\newcommand{\C}{\mathbb C}
\newcommand{\N}{\mathbb N}

\newcommand{\Her}{\operatorname{Her}}
\newcommand{\supp}{\operatorname{supp}}

\newcommand{\intt}{\operatorname{int}}

\newcommand{\sa}{\mathrm{sa}}

\newcommand{\diag}{\operatorname{diag}}

\newcommand{\MC}{\mathcal M_1(X,\Gamma)}
\newcommand{\ocap}{\operatorname{ocap}}
\newcommand{\osupp}{\operatorname{osupp}}
\newcolumntype{P}[1]{>{\raggedright\arraybackslash}p{#1}}
\newcolumntype{Y}{>{\raggedright\arraybackslash}X}
\hypersetup{pdftitle={Single generation of transformation group C*-algebras},
pdfauthor={George A. Elliott, Chun Guang Li, and Zhuang Niu},
pdfsubject={Free minimal actions of finitely generated amenable groups with URP and COS},
pdfkeywords={single generation, crossed products, Rokhlin towers, Cuntz comparison}}
\numberwithin{equation}{section}
\title[Single generation of transformation group algebras]{Single generation of transformation group\texorpdfstring{ $C^*$}{ C*}-algebras}
\author{George A. Elliott}
\address{Department of Mathematics, University of Toronto,
Toronto, Ontario M5S 2E4, Canada}
\email{elliott@math.toronto.edu}
\author{Chun Guang Li}
\address{School of Mathematics and Statistics, Northeast Normal University,
Changchun 130024, P. R. China}
\email{licg864@nenu.edu.cn}
\author{Zhuang Niu}
\address{Department of Mathematics and Statistics, University of Wyoming,
Laramie, WY 82071, USA}
\email{zniu@uwyo.edu}
\date{October 15, 2026}
\begin{document}
\begin{abstract}
We prove that $C(X)\rtimes\Gamma$ is singly generated whenever $\Gamma$
is a finitely generated discrete amenable group and its action on the
nonempty compact metrizable space $X$ is free, minimal, and has the
uniform Rokhlin property and Cuntz comparison of open sets. In
particular, every free minimal $\mathbb Z^d$-crossed product with
metrizable underlying space is singly generated.
\end{abstract}
\maketitle

\section{Introduction}

In \cite{Kadison}, Kadison asked whether every von Neumann algebra acting on a 
separable Hilbert space is singly generated (see also in \cite{Ge}). This problem
was reduced to the \(\mathrm{II}_1\) factor case(\cite{Willig}). Recently,
OpenAI \cite{OpenAI} resolves the remaining case. In the $C^*$-algebra setting,
one can ask whether each simple separable $C^*$-algebra singly generated?
Simplicity is necessary, otherwise $C(S^2)$ is not singly generated. Many classes of 
$C^*$-algebras are proved to be singly generated, such as stable $C^*$-algebras,
UHF-absorbing $C^*$-algebras, $\mathcal{Z}$-absorbing $C^*$-algebras, the first type Villadsen's algebras,
$C_r^*(\mathbb{F}_2)$, and $C_r^*(\mathbb{F}_n)$ for each $n\in\{3,4,\ldots\}\cup\{\infty\}$(\cite{OlsenZame, Topping, ThielWinter, LiNiuRuzicka,
ELN, ELN1}). On the other hand, we construct a simple separable unital AH algebra which is not singly generated(\cite{ELNNonGenerator}). Thus simplicity and separability alone do not suffice.

For transformation group algebras, the uniform Rokhlin property (URP)
and Cuntz comparison of open sets (COS) provide large matrix models and
control their complements. It is showed that these two properties
imply stable rank one for a free minimal action of a countable amenable
group \cite[Theorem 7.8]{LN}. The purpose of this paper is to prove
a single-generation theorem under the same dynamical hypotheses when
the acting group is finitely generated. In fact, we get the following result.

\begin{theorem}\label{thm:main}
Let $\Gamma$ be a finitely generated discrete amenable group, and let
$(X,\Gamma)$ be a free and minimal dynamical system, where $X$ is a
nonempty compact metrizable space. If $(X,\Gamma)$ has the URP and COS,
then $C(X)\rtimes\Gamma$ is singly generated.
\end{theorem}

Since any free and minimal $\mathbb Z^d$ action has URP and COS(\cite{NiuZd}),
we have the following result.

\begin{corollary}\label{cor:zd}
Let $d\geq1$, and let $(X,\mathbb Z^d)$ be a free and minimal dynamical
system on a nonempty compact metrizable space. Then
$C(X)\rtimes\mathbb Z^d$ is singly generated.
\end{corollary}

The proof has two stages. In the first stage, a finite family of
self-adjoint Fourier sums is recovered from two self-adjoint elements
(Proposition~\ref{prop:finite}). The interior of each Rokhlin tower
provides a matrix algebra. Its complement has uniformly small measure,
so COS and stable rank one move the remainder into a small collection
of tower rows. The available off-diagonal positions accommodate both
the Fourier coefficients and every corner of the transported
remainder. Three positions are left unused. Applied to the real and
imaginary parts of finitely many group generators, this construction
recovers all implementing unitaries.

The second stage recovers a prescribed diagonal function using a pair
in any nonzero diagonal hereditary subalgebra, relative to these known
unitaries (Proposition~\ref{prop:local}). The central reconstruction
identity is the following. If a finite family $(V_\gamma)_{\gamma\in F}$
has $S=\sum_\gamma V_\gamma^*V_\gamma$ invertible, then
\begin{equation}\label{eq:intro-frame}
 a=S^{-1}\left(\sum_{\gamma,\nu\in F}
 V_\gamma^*(V_\gamma aV_\nu^*)V_\nu\right)S^{-1}.
\end{equation}
Minimality supplies such a frame with
$V_\gamma=\chi^{1/2}u_\gamma^*$ for any nonzero
$\chi\in C(X)_+$. We encode the inner corners in
\eqref{eq:intro-frame}; the outer frame operators are already known.
This recovers the actual diagonal function without recovering a
tower-base projection. Disjoint spectral intervals then combine
countably many local pairs, and the three unused positions store the
resulting three hereditary elements.

The matrix encoding follows the principle of Olsen and Zame
\cite{OlsenZame}. Here the matrix units are multipliers, the tower algebras
need not be full corners, and the complement must be recovered in
norm. We address these points by scalar cutoffs, separate coding of
all remainder corners, and the finite-frame identity. Uniform smallness
in measure is used for comparison and is never converted into
smallness in norm.

The finite generation assumption is used to recover all implementing
unitaries in the first construction. Countably many local pairs can
then be placed in disjoint spectral intervals and combined by norm
convergent series. The method does not assume strict comparison of all
positive elements or $\mathcal Z$-stability. For comparison, under URP
and COS, zero mean dimension already gives $\mathcal Z$-stability by
theorem \cite{NiuZ}, and hence single generation by \cite{ThielWinter}.
The argument here does not impose zero mean dimension. The theorem
asserts existence of a generator; it does not assert density of
generators or a bound on generator rank.

Section~2 fixes the tower and comparison conventions. Section~3 gives
the finite-frame identity and the off-diagonal filtering lemmas.
Sections~4 and~5 establish the two recovery constructions. Section~6
combines them and proves the main theorem. Section~7 discusses the
role of finite generation and the scope of the method.

\subsection*{Acknowledgements}
The work was carried out using ChatGPT. The first author is
supported by an NSERC Discovery Grant, the second author is supported by an
NNSF Grant of China (No. 12671150), and the third author is supported by a Simons
Foundation Grant (MP-TSM-00002606).

\section{Notation and preliminaries}

\subsection{Dynamical systems and Rokhlin towers}
\begin{definition}
Consider a topological dynamical system $(X,\Gamma)$, where $X$ is a separable compact
Hausdorff space, and $\Gamma$ is a discrete group which acts on $X$ from the right. The 
dynamical system $(X,\Gamma)$ is said to be minimal if
\[
Y\gamma=Y,\quad \gamma \in\Gamma,
\]
for some closed $Y\subseteq X$ implies $Y=\emptyset$ or $Y=X$; it is said to be free if 
\[
x\gamma=x
\]
for some $x\in X$ and $\gamma\in\Gamma$ implies $\gamma=e$.
\end{definition}

\begin{definition}
A Borel measure $\mu$ on $X$ is invariant under the action $\sigma$ if for any Borel set $E\subseteq X$,
one has
\[
\mu(E)=\mu(E\gamma),\quad \gamma\in \Gamma.
\]
Denote by $\MC$ the compact space of invariant Borel probability
measures on $X$. It is a Choquet simplex under the weak* topology.
\end{definition}

\begin{definition}
Let $\Gamma$ be a countable discrete group. Let $K\subseteq \Gamma$ be a finite set and let $\varepsilon>0$.
Then a finite set $E\subseteq \Gamma$ is said to be $(K, \varepsilon)$-invariant if 
\[
\frac{|EK\Delta E|}{|E|}<\varepsilon.
\]
\end{definition}

The group $\Gamma$ is amenable if there exists a sequence $(\Gamma_n)$ of finite subsets
of $\Gamma$ such that for any $(K,\varepsilon)$, there is $N$ such that $\Gamma_n$ is $(K,\varepsilon)$-invariant 
for any $n>N$. The sequence $(\Gamma_n)$ is called a F\o lner sequence.

The $K$-interior of a finite set $E\subseteq \Gamma$ is defined as
\[
 \intt_K(E)=\{\gamma\in E:\gamma K\subseteq E\}.
\]

Consider a topological dynamical system $(X,\Gamma)$.
Then the group $\Gamma$ acts (from the left) on the $C^*$-algebra
$C(X)$ by 
\[
\alpha_\gamma(f)=f\circ\gamma.
\]
The (full) crossed product $C^*$-algebra $A=C(X)\rtimes \Gamma$ is 
defined to be the universal $C^*$-algebra generated by $C(X)$ and
unitaries $u_\gamma$, $\gamma\in\Gamma$, with respect to the relations
\[
u_\gamma fu_\gamma^*=f\circ \gamma,\quad u_{\gamma_1}u_{\gamma_2}^*=u_{\gamma_1\gamma_2^{-1}},\quad u_e=1_A, \quad \gamma,\gamma_1,\gamma_2\in\Gamma.
\] 
The $C^*$-algebra $A$ is nuclear if $\Gamma$ is amenable (see, for instance, Corollary
7.18 of \cite{Williams}). If, moreover, $(X,\Gamma)$ is minimal, then $A$ is simple.
(Theorem 5.16 of \cite{Effros} and Th\'{e}or\`{e}me 5.15 of \cite{Zeller}), i.e.
$A$ has no non-trivial closed two-sided ideals.

\begin{definition}
A topological dynamical system $(X,\Gamma)$ is said to have the uniform Rokhlin property (URP)
if for any finite set $F$, and any $\varepsilon>0$, there are open Rokhlin towers $(B_1, F_1,\ldots,(B_S, F_S)$
such that $F_1, F_2,\ldots, F_S$ are $(F,\varepsilon)$-invariant, the sets
\[
B_s\gamma,\quad \gamma\in F_s, s=1,2,\ldots,S
\]
are mutually disjoint, and
\[
\mu(X\setminus \bigsqcup_{s=1}^S\bigsqcup_{\gamma\in F_s}B_s\gamma)<\varepsilon,\quad \mu\in\MC.
\]

\end{definition}

\begin{definition}\cite{NiuZ1}
A topological dynamical system $(X,\Gamma)$ is said to have $(\lambda, m)$-Cuntz comparison
of open sets, where $\lambda\in(0,+\infty)$ and $m\in \mathbb{N}$, if, for any open sets $E,F\subseteq X$
with 
\[
\mu(E)\leq \lambda\mu(F),\quad \mu\in\MC,
\]
one has
\[
\varphi_E\precsim\varphi_F^{\oplus m}.
\]
A topological dynamical system is said to have Cuntz-comparison of open sets (COS) if 
it has $(\lambda,m)$-Cuntz-comparison of open sets for some $\lambda$ and $m$.
\end{definition}

For $a\in A_+$, put $\Her(a)=\overline{aAa}$ and let $p_a$ be its
support projection in $A^{**}$. For a positive continuous function,
write $\osupp(a)=\{x:a(x)>0\}$; $\supp(a)$ denotes closed support.
If $W\subseteq X$ is open, choose $\varphi_W\in C(X)_+$ with
$\osupp(\varphi_W)=W$, and write $\Her(W)=\Her(\varphi_W)$.
Such a function exists because $X$ is metrizable. Functions with the
same open support are Cuntz equivalent already in $C(X)$, and determine
the same hereditary subalgebra in $A$.

Indeed, for $a,b\in C(X)_+$ with
$\osupp(a)\subseteq\osupp(b)$, continuous quotients with a cutoff near
$b=0$ give $a\precsim b$ in $C(X)$.

Uniform smallness of the measure of an open support is used below only
to obtain comparison. The remainder itself is recovered exactly or by
a norm limit; no norm estimate is inferred from a measure estimate.

Throughout Sections~2--5, $\Gamma$ is countably infinite, and the action
is free and minimal with the URP and COS. The finite-group case will be
treated separately in the proof of Theorem~\ref{thm:main}.

An open Rokhlin tower $(U,E)$
has mutually disjoint levels $Ua$, $a\in E$. In a family of towers,
all levels are required to be mutually disjoint. By replacing $U$ with
$Ua_0$ and $E$ with $a_0^{-1}E$, we may assume $e\in E$.

For a right F\o lner sequence $(\Phi_n)$ and a closed set $K\subseteq X$,
write
\[
 \ocap(K)=\limsup_{n\to\infty}\sup_{x\in X}
       \frac{1}{|\Phi_n|}\sum_{\gamma\in\Phi_n}1_K(x\gamma).
\]
On closed sets this agrees with orbit capacity in \cite{LN} and satisfies
\[
 \ocap(K)=\max_{\mu\in\MC}\mu(K)
 \quad\text{\cite[Remark 2.20]{LN}}.
\]

\subsection{Preparation of towers}
\begin{lemma}\label{lem:towers}
Given a finite buffer $L$, $\varepsilon>0$, and $N\in\N$, the URP
towers may be chosen with $|E_s|\geq N$, mutually disjoint closed cores
$B_s a\subset U_s a$, and
\[
 \mu\!\left(\bigsqcup_s\bigsqcup_{a\in E_s}B_s a\right)>1-\varepsilon,
 \qquad |E_s\setminus\intt_L(E_s)|<\varepsilon |E_s|.
\]
There are functions $g_s,\xi_s,f_s\in C_0(U_s)_+$, bounded by $1$,
strictly positive $g_s$ on $U_s$, and equal to $1$ near $B_s$, with
\[
 \supp\xi_s\Subset\intt\{g_s=1\},\qquad
 \supp f_s\Subset\intt\{\xi_s=1\}.
\]
\end{lemma}

\begin{proof}
Enlarge the prescribed buffer to a finite set $L'$ containing $e$ and
more than $2N$ elements, and choose the URP error $\delta<1$ sufficiently
small. For any nonempty shape $E$ and any $a\in E$,
\[
 |L'|=|aL'|\leq|EL'|<(1+\delta)|E|,
\]
so $|E|\geq N$. Also
\[
 |E\setminus\intt_L(E)|
 \leq |L|\,|EL\setminus E|
 \leq |L|\,|EL'\setminus E|.
\]
Taking $\delta$ smaller gives the asserted interior estimate.
To obtain cores uniformly in the invariant measures, approximate the
indicator functions of the finitely many open bases increasingly by
continuous contractions of compact support. The corresponding integrals
are continuous on the compact space of invariant measures. Their sets of
values greater than a fixed number below $1$ form an increasing open
cover, so one approximation works for every measure. A level set of that
approximation gives compact cores with the required measure estimate,
after slightly relaxing the error. Discard towers whose chosen cores are empty. The remaining cores still
have the asserted measure estimate, and their translated levels are
disjoint. Urysohn functions give the nested plateaux; adding a small
strictly positive function off the plateau makes $g_s$ strictly positive
on $U_s$ without changing its value near $B_s$.
\end{proof}

\subsection{Tower algebras and exact comparison}
The tower algebra associated with $(U_s,E_s)$ is the image of
\begin{equation}\label{eq:tower-inclusion}
 \iota_s:M_{|E_s|}(C_0(U_s))\longrightarrow A,
 \qquad [c_{ab}]\longmapsto
          \sum_{a,b\in E_s}u_a^*c_{ab}u_b.
\end{equation}
It is an injective $*$-homomorphism. Disjointness of the levels gives
multiplicativity, and its restriction to each diagonal corner is
faithful. Distinct tower algebras are orthogonal.

We use multiplier matrix units $e^s_{ab}$ for this algebra, represented
in $A^{**}$ by
\[
 e^s_{ab}=u_a^*p_{U_s}u_b,
 \qquad e^s_{ab}e^s_{cd}=\delta_{bc}e^s_{ad},
\]
where $p_{U_s}$ is the characteristic projection of $U_s$.
For $c\in C_0(U_s)$, the notation $ce^s_{ab}$ means the actual element
$u_a^*cu_b\in A$. A scalar base function acts as the corresponding
central multiplier of the whole matrix block, not just as a function
on its identity level. In particular, every use of a multiplier below
is accompanied by a coefficient vanishing at the boundary of its base.

\begin{lemma}\label{lem:membership}
Let $C=M_n(C_0(U))$ be a tower algebra, with its multiplier algebra
represented in $A^{**}$. A bounded continuous matrix coefficient
$b\in M(C)$ satisfies $cb,bc\in C\subseteq A$ for every scalar
$c\in C_0(U)$. If $cb=p_\chi(cb)p_\chi$, then
$cb\in\Her(\chi)$.

More generally, let $C_s,C_t$ be two tower algebras, let $T_s,T_t$
be positive scalar elements in them, and let $v_s,v_t$ be bounded multiplier
columns. If $d\in A$ satisfies $T_sd=d=dT_t$, then
\[
 v_s^*dv_t=(T_sv_s)^*d(T_tv_t)\in A.
\]
The same observation applies to translating this coefficient to other
columns. It requires no multiplier matrix unit between different towers.
\end{lemma}

\begin{proof}
A bounded continuous matrix function multiplied by a function in
$C_0(U)$ again vanishes at infinity. The hereditary assertion follows
from $\Her(\chi)=p_\chi A^{**}p_\chi\cap A$. For the last formula,
$T_sv_s$ and $T_tv_t$ are actual elements of $A$, and scalar
multiplier functions commute with the columns.
\end{proof}

For clarity, $p_{s,a}$ will always mean the multiplier projection
$e^s_{aa}$ of a tower row. Such a projection may be used to describe
support, but its membership in a generated subalgebra will never be
assumed. We call $t_se^s_{ab}$ a \emph{soft matrix unit}: it is an
actual element of $A$, although $e^s_{ab}$ need not be.

\begin{lemma}\label{lem:exact}
The algebra $A$ has stable rank one. If $a,b\in A_+$ and $a\precsim b$,
then there is $v\in A$ with $v^*v=a$ and $vv^*\in\Her(b)$.
If $a,b\in\Her(\chi)$, then $v\in\Her(\chi)$.
\end{lemma}

\begin{proof}
Stable rank one is \cite[Theorem 7.8]{LN}. We deduce the required
implementation from the Hilbert-module description of the Cuntz
semigroup of Coward, Elliott, and Ivanescu
\cite[Theorem 3 and Appendix 6]{CEI}. Put
\[
 E_a=\overline{aA},\qquad E_b=\overline{bA},\qquad
 \langle x,y\rangle=x^*y.
\]
These are countably generated Hilbert $A$-modules, and $a\precsim b$
means $[E_a]\leq[E_b]$ in the Hilbert-module picture. Since $A$ has
stable rank one, the cited theorem says that this order is realized
by an inner-product preserving $A$-linear embedding $U:E_a\to E_b$.
The element $a^{1/2}$ belongs to $E_a$ by continuous functional
calculus. Hence $v=U(a^{1/2})\in E_b\subseteq A$ satisfies
\[
 v^*v=\langle U(a^{1/2}),U(a^{1/2})\rangle
     =\langle a^{1/2},a^{1/2}\rangle=a.
\]
Choose $x_n\in A$ with $bx_n\to v$. Then
$bx_nx_n^*b\to vv^*$, so $vv^*\in\overline{bAb}=\Her(b)$.
No adjointability of the module embedding is needed.
For the last assertion, both $v^*v$ and $vv^*$ belong to
$\Her(\chi)$. An approximate unit
$(e_n)$ of $\Her(\chi)$ satisfies $e_nv\to v$ and $ve_n\to v$,
since
\[
 \|(1-e_n)v\|^2
   =\|(1-e_n)vv^*(1-e_n)\|\longrightarrow0
\]
and the analogous identity uses $v^*v$ on the other side. Hence
$e_nve_n\to v$.
\end{proof}

\begin{lemma}\label{lem:division}
Let $t$ be a strictly positive contraction in a scalar tower algebra.
If a continuous bounded multiplier coefficient $a$ satisfies
$\|a(x)\|\leq Ct(x)$, then
$\theta_\varepsilon(t)a\to a$ in norm, where $\theta_\varepsilon$ is a
smooth function taking values in $[0,1]$ and
$\theta_\varepsilon=0$ on $[0,\varepsilon]$ and $1$ on
$[2\varepsilon,1]$. Thus expressions involving inverse powers of $t$
can be recovered with continuous functions which vanish near zero,
provided the final coefficient is $O(t)$.
If $c\in\Her(r)$ and $r$ and $rcr$ are recovered, then $c$
is recovered by
\[
 k_\varepsilon(r)(rcr)k_\varepsilon(r)
 =\theta_\varepsilon(r)c\theta_\varepsilon(r)\longrightarrow c,
 \qquad k_\varepsilon(t)=\theta_\varepsilon(t)/t.
\]
\end{lemma}

\begin{proof}
The first error is at most $2C\varepsilon$. The second convergence is
the approximate-unit property of $\Her(r)$.
\end{proof}

\section{Finite frames and spectral recovery}

\begin{lemma}\label{lem:frame}
Let $A$ be unital and let $D\subseteq A$ be a unital $C^*$-subalgebra.
Suppose $V_1,\ldots,V_M\in D$ and
$S=\sum_jV_j^*V_j$ is invertible in $A$. If $a\in A$ and
$V_jaV_k^*\in D$ for every $j,k$, then $a\in D$, with
\[
 a=S^{-1}\left(\sum_{j,k}V_j^*(V_jaV_k^*)V_k\right)S^{-1}.
\]
For a minimal action and $0\ne\chi\in C(X)_+$, a finite family
$V_\gamma=\chi^{1/2}u_\gamma^*$ always has this property.
\end{lemma}

\begin{proof}
Spectral permanence gives $S^{-1}\in D$, and the middle sum equals
$SaS$. For the last assertion, translates of $\{\chi>0\}$ cover $X$
by minimality. A finite subcover, enlarged to a symmetric finite set
$F$ containing the identity, gives
$S=\sum_{\gamma\in F}\alpha_\gamma(\chi)>0$ everywhere. Compactness
then gives $S\geq\kappa1$ for some $\kappa>0$.
\end{proof}

\begin{lemma}\label{lem:filter}
Let $H=H^*$ in a unital $C^*$-algebra and $x$ be any element.
For $F\in C_c^\infty(\mathbb R^2)$ define
\[
 \Phi_F^H(x)=\frac1{(2\pi)^2}\iint \widehat F(s,t)
       e^{isH}xe^{itH}\,ds\,dt,
 \quad
 \widehat F(s,t)=\iint F(a,b)e^{-i(sa+tb)}\,da\,db.
\]
Then $\Phi_F^H(x)\in C^*(1,H,x)$. If, in a tower block,
$H=t\diag(\lambda_i)$ and $x=a e_{ij}$, then
\[
 \Phi_F^H(x)=F(\lambda_it,\lambda_jt)x.
\]
If $Hx=\lambda x$ and $x=xp_J(H)$ for a spectral interval $J$, then
$\Phi_F^H(x)=xF(\lambda,H)p_J(H)$.
\end{lemma}

\begin{proof}
The Fourier transform is integrable, so the displayed integral is a norm
Bochner integral of elements of $C^*(1,H,x)$. The two formulas follow
from Fourier inversion, first in the scalar tower algebra and then in
the continuous functional calculus of the right copy of $H$.
\end{proof}

\begin{lemma}\label{lem:labels}
\begin{enumerate}
\item A finite family of labels $\lambda_1,\ldots,\lambda_N$ in
$(2,2.1)$ can be chosen so that all ratios
$\lambda_i/\lambda_j$, $i\ne j$, are distinct.
\item Let $C=M_N(C_0(U))$ be a tower algebra, let $t\in C_0(U)$
be a strictly positive contraction, and suppose that $H=H^*\in A$ acts on this
block as $t\diag(\lambda_i)$. Let
\[
 x=\sum_{(i,j)\in J}a_{ij}e_{ij},\qquad
 J\subseteq\{(i,j):i\ne j\},
\]
with distinct ordered positions, scalar coefficients $a_{ij}\in C_0(U)$,
and $|a_{ij}|\leq C_{ij}t$. Then each $a_{ij}e_{ij}$ belongs to
$C^*(1,H,x)$. The same conclusion holds for finitely many orthogonal
tower blocks if the labels are chosen together.
\end{enumerate}
The statement in (2) concerns off-diagonal positions only.
\end{lemma}

\begin{proof}
For (1), choose
\[
 \lambda_i=2\exp(c2^i),\qquad
 c=\frac{\log(21/20)}{2^{N+1}},\qquad 1\leq i\leq N.
\]
These labels lie in $(2,2.1)$. The nonzero differences $2^i-2^j$
are distinct for distinct ordered pairs: for a positive difference,
its power of $2$ determines the smaller index, and its odd part
determines the larger index. Negative differences are treated by
reversing the pair. This proves the ratio assertion.

For (2), the rays
$R_{ij}=\{(\lambda_i\tau,\lambda_j\tau):0\leq\tau\leq1\}$
meet only at the origin. Fix $\varepsilon>0$. A smooth compactly
supported function on $\mathbb R^2$ can be chosen equal to
$\theta_\varepsilon(\tau)$ on the selected ray and zero on all other
rays, where $\theta_\varepsilon$ is the cutoff of
Lemma~\ref{lem:division}. One obtains it by a smooth extension along
the ray and a narrow tubular cutoff on its portion away from zero.
Lemma~\ref{lem:filter} gives
$\theta_\varepsilon(t)a_{ij}e_{ij}\in C^*(1,H,x)$.
Its distance from $a_{ij}e_{ij}$ is at most
$2C_{ij}\varepsilon$. The finite-block assertion uses the same
construction with the globally distinct ratios.
\end{proof}

\begin{lemma}\label{lem:soft}
Let $(e_{ij})_{i,j=1}^N$, $N\geq2$, be multiplier matrix units of a matrix algebra over
$C_0(U)$, and let $t\in C_0(U)$ be a strictly positive contraction.
If a $C^*$-subalgebra $D\subseteq A$ contains $te_{ij}$ for the edges
of a connected graph on the rows, then it contains every $te_{ij}$
and every $f(t)e_{ij}$ with $f$ continuous and $f(0)=0$.
\end{lemma}

\begin{proof}
For an edge $x=te_{ij}$, the positive square roots of $x^*x$ and
$xx^*$ are $te_{jj}$ and $te_{ii}$. Products along a path give
$t^ke_{ij}$. Multiplying by the continuous function
$\theta_\varepsilon(t)t^{1-k}$ of the known diagonal $te_{ii}$,
extended by zero near zero, gives
$\theta_\varepsilon(t)te_{ij}\to te_{ij}$ in norm.
Continuous functional calculus recovers $f(t)e_{ii}$.
Multiplication by
$k_\varepsilon(te_{ii})(te_{ij})$, where
$k_\varepsilon(\tau)=\theta_\varepsilon(\tau)/\tau$, then gives
$\theta_\varepsilon(t)f(t)e_{ij}\to f(t)e_{ij}$.
The last convergence uses continuity at zero, rather than a
boundedness assumption on $f(t)/t$.
\end{proof}

\begin{remark}\label{rem:filter-caution}
We use smooth functions with integrable Fourier transforms, not an
unjustified continuous two-variable functional calculus on $A$.
No uniform bound on the filters as $\varepsilon\downarrow0$ is needed:
at each stage all unwanted entries are annihilated exactly and the
remaining scalar cutoff error tends to zero in norm. Fixed point
entries are assigned positions off every occupied ray, including the
reverse rays. Entries on a vertical or horizontal axis are recovered
later by row cutoffs. Diagonal entries all lie on the ray of ratio
$1$, so the off-diagonal separation lemma makes no assertion about
separating them. In the constructions below, every diagonal
remainder corner is first moved to an unused off-diagonal position.
\end{remark}

\section{Recovery of a finite Fourier family}

\begin{proposition}\label{prop:finite}
Let $\Gamma$ be a countably infinite amenable group acting freely and
minimally on the compact metrizable space $X$, and suppose the action has
URP and COS. For every finite family $a_1,\ldots,a_q$ of self-adjoint
finite Fourier sums in $A$, there are $H_0\in C(X)_+$ and
$K_0\in A_\sa$ such that $D_0=C^*(1,H_0,K_0)$ contains all the $a_i$.
There is a nonempty open set $W\subset X$ such that, for every
$P_1,P_2,P_3\in\Her(W)_\sa$, one can replace $K_0$ by a
self-adjoint $K\in A$ satisfying
\[
 D_0\subset C^*(1,H_0,K),\qquad
 P_j\in C^*(1,H_0,K)\quad(j=1,2,3).
\]
The three additional entries occupy distinct previously unused
off-diagonal matrix positions of one tower block.
\end{proposition}

\begin{proof}
We give the construction and then verify recovery in its required
order. All matrix units below are multiplier matrix units; every
expression asserted to belong to $A$ has a coefficient in the
appropriate $C_0(U_s)$.
\smallskip\noindent\emph{(1) Compression and comparison.}
Enlarge a common Fourier support to a finite symmetric set $L$ containing
the identity, and write $\ell=|L|$. We may suppose $q\geq1$. Let
$\lambda,m$ be the COS constants. Fix
\[
 0<\rho\leq\frac{1}{100\sqrt{q+1}},\qquad
 0<\delta<\min\left\{\frac1{100},\frac\rho{100},
                         \frac{\lambda\rho}{100m}\right\},
 \qquad 0<\eta<\frac1{100}.
\]
Choose the towers of Lemma~\ref{lem:towers} with buffer $L$, error
$\delta$, and shapes of size
\[
 n_s=|E_s|\geq N,\qquad
 N\geq100(q\ell+q+10)+\frac{2m}{\rho}.
\]
Put $I_s=\intt_L(E_s)$ and $d_s=|I_s|$. In their tower algebra
$B=\bigoplus_s M_{n_s}(C_0(U_s))$, set
\[
 h=\sum_s\sum_{a\in I_s}g_s e^s_{aa},\qquad r=1-h,
 \qquad c_i=r^{1/2}a_i r^{1/2},\qquad b_i=a_i-c_i.
\]
The diagonal functions in this formula are extended by zero off their
tower levels. Since $h=1$ on every interior closed core,
\begin{equation}\label{eq:finite-small}
 \mu(\{r>0\})<2\delta
 \quad\text{for every invariant probability measure }\mu.
\end{equation}
Indeed, discarding boundary cores loses at most
$\delta\sum_s n_s\mu(B_s)\leq\delta$ in addition to the core error.

We have $b_i\in B$ exactly. For a Fourier term $f_\gamma u_\gamma$,
the coefficient of its difference from the compression by $r^{1/2}$ is
\[
 f_\gamma\bigl(1-r^{1/2}\alpha_\gamma(r^{1/2})\bigr).
\]
Whenever this coefficient is nonzero, one endpoint is an $L$-interior
level. Symmetry of $L$ places the other endpoint in the same tower:
if the range row is $a$, its source row is $a\gamma$; conversely, a
source row $b$ has range row $b\gamma^{-1}$. In base coordinates the scalar multiplier is $0$, $g_s$, or
$1-\sqrt{1-g_s}$. Since $1-\sqrt{1-g_s}\leq g_s$, all three are
bounded by $g_s$. Consequently each
block of each $b_i$ has at most $\ell n_s$ possibly nonzero matrix
coefficients, and every such coefficient $\beta$ satisfies
$|\beta|\leq C_i g_s$ for a fixed finite constant $C_i$.

In each $I_s$ reserve $k_s=m\lfloor\rho n_s/m\rfloor$ rows and divide
them into $m$ groups of equal size $k_s/m$. The groups in different
blocks are combined using the same group index. Let $V$ be the Cartan
function obtained by putting $f_s$ in the first group of each block,
and let $R$ put $f_s$ in all reserved rows. For each invariant $\mu$,
\[
 \mu(\{V>0\})\geq\frac\rho{2m}\sum_s n_s\mu(B_s)
 >\frac\rho{2m}(1-\delta).
\]
Our choices imply $2\delta<\lambda\rho(1-\delta)/(2m)$.
Thus COS and~\eqref{eq:finite-small} give $r\precsim V^{\oplus m}$.
The $m$ row groups are identical copies of the scalar base functions,
so $V^{\oplus m}$ is Cuntz equivalent to $R$. More explicitly, the
row identifications, multiplied by $f_s^{1/2}$, form a rectangular
matrix over $A$ whose source square is $V^{\oplus m}$ and whose range
square is $R$. Thus the stabilized comparison has a single-element
target $R\in A$. No bare multiplier is asserted to belong to $A$. Lemma~\ref{lem:exact} therefore provides $v\in A$ with
\begin{equation}\label{eq:finite-v}
 v^*v=r,\qquad vv^*\in\Her(R).
\end{equation}
Put $d_i=vc_iv^*\in\Her(R)$.

Use the nested cutoffs of Lemma~\ref{lem:towers} and put
\[
 t_s=\eta g_s+(1-\eta)\xi_s.
\]
This function is strictly positive on $U_s$, belongs to $C_0(U_s)$,
and equals $1$ on $\supp f_s$. Fix a tower $s_*$ with nonempty core,
and require that the largest high label occur in this tower. Choose
globally distinct high labels
$\lambda_{s,a}\in[2,2.1]$, $a\in I_s$, with all nontrivial ordered
ratios distinct, as in Lemma~\ref{lem:labels}. Give each boundary row
a distinct positive label $\nu_{s,b}<\eta$, and set
\begin{equation}\label{eq:finite-H}
 H_0=\sum_s t_s\left(\sum_{a\in I_s}\lambda_{s,a}e^s_{aa}
                 +\sum_{b\in E_s\setminus I_s}\nu_{s,b}e^s_{bb}\right).
\end{equation}
On the support of $r$, every high row has $g_s<1$, hence $\xi_s=0$,
and $t_s=\eta g_s$. Thus
\begin{equation}\label{eq:finite-spectral}
 (H_0-2.1\eta)_+r^{1/2}=0,\qquad
 p_r\leq1_{[0,2.1\eta]}(H_0).
\end{equation}
On a reserved row supporting $vv^*$, however, $t_s=1$ and $H_0$ has
the fixed corresponding high label.
\smallskip\noindent\emph{(2) Assignment of the matrix coefficients.}
We now assign the entries of $K_0$. An unordered high--high position
means a pair of distinct high rows; both orientations of that position
are occupied together. First choose, in each block, a chain through all
its high rows, and put $t_s e^s_{a,a'}$ on its successive edges. For
each actual scalar coefficient $\beta$ of each $b_i$, assign a new
high--high position in the same block and put
$t_s\beta e^s_{p,q}$ there. Assign one further position to
$t_sg_s e^s_{p,q}$.

For the entries of $d_i$, write $Q_{s,a}=\xi_s e^s_{aa}$ on each
reserved row. Since $\xi_s=1$ on $\supp f_s$, the actual corners
\[
 d_i^{s,a;t,b}=Q_{s,a}d_iQ_{t,b}
\]
sum to $d_i$ over all ordered reserved row pairs. For an unordered
pair of different reserved rows, encode one corner, obtaining its
adjoint automatically. Encode each diagonal corner separately in an
off-diagonal position. Choose a new position $(p,q)$ between the same
two tower blocks (within that block for a diagonal corner), and put
\begin{equation}\label{eq:finite-corner-code}
 T_i^{s,a;t,b}
       =(t_s e^s_{p,a})\,d_i^{s,a;t,b}\,(t_t e^t_{b,q})
\end{equation}
there. These are actual elements of $A$. Their endpoint functions
equal $1$ on the support of the original corner, so
\[
 H_0 T_i^{s,a;t,b}=\lambda_{s,p}T_i^{s,a;t,b},\qquad
 T_i^{s,a;t,b}H_0=\lambda_{t,q}T_i^{s,a;t,b}.
\]
Only the two endpoint labels specify a position between different
blocks; no multiplier matrix unit between those blocks is used.

Here is an explicit check that these assignments are possible. Within
block $s$ they require at most
\begin{equation}\label{eq:finite-capacity}
 (d_s-1)+q\ell n_s+1+\frac{qk_s(k_s+1)}2+3
 \leq(q\ell+1)n_s+\frac{q\rho^2n_s^2+q\rho n_s}{2}+3
\end{equation}
unordered positions, including the three positions we leave vacant.
The available number is $d_s(d_s-1)/2$, with
$d_s>(1-\delta)n_s$. Since $q\rho^2\leq10^{-4}$,
$q\rho\leq q/100$, and $n_s\geq100(q\ell+q+10)$,
the right side of~\eqref{eq:finite-capacity} is less than
$0.011n_s^2$, whereas $d_s(d_s-1)/2>0.48n_s^2$.
Between different blocks the number required is
$qk_sk_t\leq10^{-4}n_sn_t$, whereas there are
$d_sd_t>0.98n_sn_t$ available positions. Thus every assignment can be
made injectively. Fix three vacant unordered positions in block $s_*$.

For every boundary row $b\in E_s\setminus I_s$, choose a distinct
high target row $j(b)\in I_s$ outside the reserved rows, and put
$t_se^s_{j(b),b}$ in the packet. This is possible because
$|E_s\setminus I_s|<\delta n_s$ and
$d_s-k_s>(1-\delta-\rho)n_s$. Let $S$ be the sum of all the chain
entries, scalar entries, corner entries, and boundary entries just
described, using one orientation for each assigned unordered
position. Define
\[
 K_0=S+S^*+v+v^*\in A_\sa.
\]
\smallskip\noindent\emph{(3) Spectral separation and recovery.}
We verify recovery from $D_0=C^*(1,H_0,K_0)$. For a high label pair
$i,j$, write
$R_{ij}=\{(\lambda_i\tau,\lambda_j\tau):0\leq\tau\leq1\}$.
Here and in Table~\ref{tab:finite-types}, a joint spectral type means
that the formulas of Lemma~\ref{lem:filter} annihilate the indicated
entry whenever the filter vanishes on the displayed set.

\begin{table}[htbp]
\centering\small\renewcommand{\arraystretch}{1.12}
\caption{Spectral types in the finite construction. Each adjoint has
the coordinates reversed.}\label{tab:finite-types}
\begin{tabularx}{\textwidth}{@{}P{.24\textwidth}P{.40\textwidth}Y@{}}
\toprule
Entry & Joint type for $H_0$ & Recovery \\
\midrule
Chain and scalar data
& $R_{ij}$, with $i\ne j$
& Smooth ray filter; origin cutoff \\
Moved remainder corner
& $\{(\lambda_i,\lambda_j)\}$, $i\ne j$
& Smooth point filter \\
Boundary link
& $\{(\lambda_j\tau,\nu_b\tau):0\leq\tau\leq1\}$
& High row cutoff after subtraction \\
Reserved row of $v$
& $\{\lambda_i\}\times[0,2.1\eta]$
& Reserved row and high spectral cutoff \\
Hereditary insertion
& $\{(\lambda_{p_j},\lambda_{q_j})\}$
& Smooth point filter in a vacant position \\
\bottomrule
\end{tabularx}
\end{table}

The first two types have coordinate ratios in $[2/2.1,2.1/2]$.
Distinct assigned ordered positions have distinct ratios, including
the reverse orientations. Boundary rays have ratios greater than
$2/\eta$, and their adjoints have ratios less than $\eta/2$.
For a component of $v$, its range is on a reserved plateau row and
its source is in $1_{[0,2.1\eta]}(H_0)$, by
\eqref{eq:finite-spectral}. Thus the mixed rectangles in the table
are disjoint from every high--high ray and point. Indeed, a
high--high ray with first coordinate at least $2$ has second
coordinate at least $4/2.1>2.1\eta$.

Fix a scalar or chain entry and $\varepsilon>0$. Its ray segment
away from the origin is separated by a positive distance from
every other type. Choose the smooth joint filter equal to
$\theta_\varepsilon(\tau)$ on that ray and zero on every unwanted
type. The resulting element is the selected entry times
$\theta_\varepsilon(t_s)$. Since its coefficient is $O(t_s)$,
Lemma~\ref{lem:division} gives convergence to the entry in norm.
For a moved corner, its single point has positive distance from all
other types; a filter equal to $1$ at that point recovers it exactly.
Consequently all high--high entries, and their finite self-adjoint
sum $K_{\mathrm{hh}}$, belong to $D_0$. Boundary links and comparison
links may overlap each other; they are recovered subsequently,
in the order below.

The chain entries give every $t_se^s_{ab}$ for $a,b\in I_s$, by
Lemma~\ref{lem:soft}. The coded entry $t_sg_se^s_{pq}$ gives
$g_se^s_{aa}$ for every high row $a$: multiply on the two sides by
the known soft matrix units to obtain $t_s^3g_se^s_{aa}$, and use
continuous inverse cutoffs of $t_se^s_{aa}$. The error tends to zero
because $g_s\leq t_s/\eta$. Thus
\[
 Q_{s,a}=\frac{t_se^s_{aa}-\eta g_se^s_{aa}}{1-\eta}\in D_0
 \quad(a\in I_s).
\]
This argument uses only high rows and precedes recovery of $v$.

Choose a scalar continuous function $\chi$ equal to zero on
$[0,2.1\eta]$ and to one on $[2,2.1]$. On a reserved row the
cutoff $Q_{s,a}$ equals the row projection on the range support of
$v$. Its product with every boundary entry vanishes because boundary
targets lie outside the reserved rows. Also
$\chi(H_0)v^*=0$, by~\eqref{eq:finite-spectral}. Consequently
\begin{equation}\label{eq:finite-recover-v}
 Q_{s,a}\chi(H_0)(K_0-K_{\mathrm{hh}})
       =p_{s,a}v\in D_0,
 \qquad
 v=\sum_{s,a\ \mathrm{reserved}}p_{s,a}v\in D_0.
\end{equation}
The projection notation on the right merely identifies the recovered
element; the expression on the left consists of actual known elements.

Subtract $v+v^*$ as well. The remaining packet consists just of the
boundary links and their adjoints. Multiplication on the left by
$t_se^s_{j(b),j(b)}$ isolates
$t_s^2e^s_{j(b),b}$, because the high targets are distinct. Inverse
cutoffs of that known diagonal recover $t_se^s_{j(b),b}$. Its modulus
gives $t_se^s_{bb}$. The high chain and these boundary links form a
connected graph through every row of the block, so products and
inverse cutoffs give
\begin{equation}\label{eq:finite-all-soft}
 t_se^s_{ab}\in D_0\quad(a,b\in E_s).
\end{equation}

We can now decode every original coefficient of $b_i$. If
$t_s\beta e^s_{pq}$ was its coded entry and its original position was
$(a,b)$, multiplication by the two known soft matrix units gives
$t_s^3\beta e^s_{ab}$. Inverse cutoffs on the left diagonal recover
$\beta e^s_{ab}$ in norm, since $\beta=O(g_s)=O(t_s)$. Summing gives
$b_i\in D_0$. Similarly, multiplying
\eqref{eq:finite-corner-code} by $t_se^s_{a,p}$ and
$t_te^t_{q,b}$ recovers $d_i^{s,a;t,b}$ exactly: the endpoint
functions equal $1$ on its support. Include the adjoints for the
opposite corners and sum to obtain $d_i\in D_0$.
Finally $r=v^*v\in D_0$ and
\[
 v^*d_iv=rc_ir\in D_0.
\]
To complete the norm recovery, use the actual functions
$k_\varepsilon(\tau)=\theta_\varepsilon(\tau)/\tau$, extended by zero
near zero. Then
\begin{equation}\label{eq:finite-remainder-limit}
 k_\varepsilon(r)(v^*d_iv)k_\varepsilon(r)
   =\theta_\varepsilon(r)c_i\theta_\varepsilon(r)
      \longrightarrow c_i
 \quad\text{in norm}.
\end{equation}
This follows from $c_i=r^{1/2}a_ir^{1/2}\in\Her(r)$ and the
approximate-unit property of $\theta_\varepsilon(r)$.
Thus $a_i=b_i+c_i\in D_0$. The measure estimate
\eqref{eq:finite-small} was used only to obtain $v$; the recovery
in \eqref{eq:finite-remainder-limit} is a norm argument.
\smallskip\noindent\emph{(4) Three additional hereditary elements.}
It remains to verify the three additional positions. Let $w$ be the
row in block $s_*$ with the largest high label. Choose a nonempty
open set $W$ whose closure lies in the $w$-level of
$\intt\{\xi_{s_*}=1\}$. In particular, $t_{s_*}=1$ on $\overline W$.
Give each of the three vacant positions an orientation $(p_j,q_j)$,
$j=1,2,3$, and define the known actual elements
\[
 A_j=t_{s_*}e^{s_*}_{p_j,w},\qquad
 B_j=t_{s_*}e^{s_*}_{w,q_j}.
\]
For any prescribed $P_j\in\Her(W)_\sa$, put
\[
 E_j=A_jP_jB_j,
 \qquad K=K_0+\sum_{j=1}^3(E_j+E_j^*).
\]
The support plateau gives
\[
 H_0E_j=\lambda_{s_*,p_j}E_j,
 \qquad E_jH_0=\lambda_{s_*,q_j}E_j.
\]
The point $(\lambda_{s_*,p_j},\lambda_{s_*,q_j})$ lies on no old
occupied high--high ray or point, since its unordered position was
vacant and all ordered ratios are distinct. It also lies off every
boundary ray and every mixed spectral set of $v+v^*$. It differs
from the three reverse points and the other new points. A smooth
joint filter, equal to one near this point and zero on all the other
spectral types, recovers $E_j$ exactly from $H_0,K$.
Subtracting all six entries recovers $K_0$, hence $D_0$, and
\[
 A_j^*E_jB_j^*
     =(t_{s_*}^2e^{s_*}_{ww})P_j
                           (t_{s_*}^2e^{s_*}_{ww})=P_j.
\]
This proves the proposition.
\end{proof}

\section{Local recovery of diagonal functions}\label{sec:local}

The group unitaries are assumed to be available in this section.  This
assumption is used twice: in a finite frame, and when recovered base
functions are moved back to their original tower levels.  The argument
does not require a hereditary subalgebra to inherit URP.

\begin{lemma}\label{lem:source-support}
Let $F\subseteq\Gamma$ be finite and symmetric. In a finite family
of towers $(U_s,E_s)$ choose centers $c_{s,l}$ with
$c_{s,l}F^2\subseteq E_s$, and orthogonal multiplier projections
$q_{s,l}$ whose matrix rows are supported in $c_{s,l}F$.
Let $g_s,\xi_s$ be the nested tower functions, put
\[
 h=\bigoplus_sg_s1_{|E_s|},\quad r=1-h,\quad
 t_s=\eta g_s+(1-\eta)\xi_s,\quad
 H=\sum_s t_s\sum_l\lambda_{s,l}q_{s,l},
\]
where $0<\eta<1$ and $0<\lambda_{s,l}\leq b$.
If $e_\gamma\in C(X)_+$ vanishes wherever
$\alpha_\gamma^{-1}(r)$ vanishes, then, for $\gamma\in F$,
\begin{equation}\label{eq:buffer-support}
 (H-b\eta)_+e_\gamma^{1/2}=0,\qquad
 p_{e_\gamma}\leq1_{[0,b\eta]}(H).
\end{equation}
\end{lemma}

\begin{proof}
Compute in the original tower matrix coordinates. At $y\in U_s$
with $g_s(y)<1$, the nesting of the plateaux implies $\xi_s(y)=0$.
Hence $t_s(y)\leq\eta$ and $0\leq H(y)\leq b\eta1$; the positive
part in \eqref{eq:buffer-support} is zero on the whole fibre.
If $g_s(y)=1$, a selected row is $yc_{s,l}\delta$ with
$\delta\in F$. For $\gamma\in F$,
\[
 c_{s,l}\delta\gamma^{-1}\in c_{s,l}F^2\subseteq E_s,\qquad
 r(yc_{s,l}\delta\gamma^{-1})=0.
\]
Thus $e_\gamma$ vanishes on every row supporting any $q_{s,l}$.
In this fibre, $q_{s,l}e_\gamma^{1/2}=0$, and consequently
$(H-b\eta)_+e_\gamma^{1/2}=0$. Outside the tower levels, $H=0$.
The products just computed are actual elements of the tower
algebras: multiplication of their bounded matrix coefficients by
a function vanishing with $t_s$ is covered by
Lemma~\ref{lem:membership}. Faithfulness of the tower embeddings
therefore proves the first identity in $A$.

In $A^{**}$, the closure of the range of $e_\gamma^{1/2}$ lies in
the kernel of $(H-b\eta)_+$. Since $H\geq0$, the projection onto
this kernel is $1_{[0,b\eta]}(H)$. This proves the second assertion.
The argument establishes annihilation, not only a bound on a
compression of $H$.
\end{proof}

\begin{proposition}\label{prop:local}
Let $0\ne\chi\in C(X)_+$ and let $f\in C(X)_{\sa}$.
There are $H\in\Her(\chi)_+$ and $K\in\Her(\chi)_{\sa}$ such that
\[
 f\in C^*(1,\chi,\{u_\gamma:\gamma\in\Gamma\},H,K).
\]
\end{proposition}

\begin{proof}
Write
\[
 D_0=C^*(1,\chi,\{u_\gamma:\gamma\in\Gamma\}).
\]
Replacing $\chi$ by a scalar multiple, assume $\|\chi\|\leq1$.
Minimality and compactness give a finite symmetric set
$F\subseteq\Gamma$, containing $e$, such that
\begin{equation}\label{eq:local-frame}
 V_\gamma=\chi^{1/2}u_\gamma^*,\qquad
 S=\sum_{\gamma\in F}V_\gamma^*V_\gamma
   =\sum_{\gamma\in F}\alpha_\gamma(\chi)\geq\kappa1
 \quad(\kappa>0).
\end{equation}
In particular, $V_\gamma,S^{-1}\in D_0$.  Put $M=|F|$ and
$C_F=2|FF^{-1}|$.  The functions
\begin{equation}\label{eq:local-partition}
 a_\gamma=\alpha_\gamma(\chi)S^{-1}\quad(\gamma\in F)
\end{equation}
belong to $D_0\cap C(X)$, are positive, and sum to $1$.

We first specify the numerical room needed later.  Fix $0<\eta<10^{-3}$
and $0<\rho<10^{-2}$.  If $(\lambda,m)$ are the COS constants, choose
\begin{equation}\label{eq:local-error-choice}
 0<\varepsilon<\min\left\{\frac14,
              \frac{\lambda\rho}{8C_FMm}\right\}.
\end{equation}
The tower heights will be taken large enough to satisfy all the finite
inequalities below.  This is possible in Lemma~\ref{lem:towers}, since
the constants $F,M,C_F,\rho,m$ have already been fixed.
For example, it suffices to require $N_s\geq C_FR_*$, where
\[
 R_* =\max\left\{\frac{2Mm}{\rho},\,4(M+1),\,
                        16(C_F+M+2)\right\}.
\]

Apply Lemma~\ref{lem:towers} with buffer $F^2$, error
$\varepsilon$, and the prescribed minimum height. Since
$\varepsilon<1/4$, its interior estimate gives
\[
 |\intt_{F^2}(E_s)|\geq\tfrac12|E_s|.
\]
Use its mutually disjoint towers $(U_s,E_s)$, closed cores $B_s$, and
nested functions $g_s,\xi_s,f_s$.  Write $N_s=|E_s|$.  Let
$C_s\cong M_{N_s}(C_0(U_s))$ be the original tower algebra, with
multiplier matrix units $e^s_{ab}$, $a,b\in E_s$.  In this identification
a point $y\in U_s$ represents all the physical levels $ya$,
$a\in E_s$.  Define
\begin{equation}\label{eq:local-global-cutoffs}
 h=\bigoplus_s g_s1_{N_s},\qquad r=1-h,
 \qquad t_s=\eta g_s+(1-\eta)\xi_s,
 \qquad T_s=t_s1_{N_s}\in C_s.
\end{equation}
The direct-sum notation denotes the inclusions into $A$, and $h$ is a
diagonal function on $X$.  The function $g_s$ is strictly positive on
$U_s$, and
\begin{equation}\label{eq:local-t-properties}
 t_s\geq\eta g_s>0\text{ on }U_s,\qquad
 t_s\leq\eta\text{ where }g_s<1.
\end{equation}
Moreover, $t_s=1$ on an open neighborhood of $\supp f_s$, containing
$B_s$.  The open support $E=\{r>0\}$ satisfies
\begin{equation}\label{eq:local-small-remainder}
 \mu(E)<\varepsilon\qquad(\mu\in \MC),
\end{equation}
because $h=1$ on all the closed core levels.

The positive functions and the remainder functions to be transported
are
\begin{equation}\label{eq:local-transported-remainders}
 e_\gamma=\alpha_\gamma^{-1}(a_\gamma r)
   =\chi\,\frac{\alpha_\gamma^{-1}(r)}
                    {\alpha_\gamma^{-1}(S)},\qquad
 c_\gamma^{(0)}=e_\gamma\alpha_\gamma^{-1}(f)
 \quad(\gamma\in F).
\end{equation}
Thus $e_\gamma,c_\gamma^{(0)}\in\Her(\chi)$,
$c_\gamma^{(0)}\in\Her(e_\gamma)$, and
\begin{equation}\label{eq:local-exact-diagonal-decomposition}
 f=hf+\sum_{\gamma\in F}\alpha_\gamma(c_\gamma^{(0)}).
\end{equation}
By invariance, the open support of each $e_\gamma$ has measure less than
$\varepsilon$ for every invariant measure.

\medskip\noindent
\emph{(1) Return columns inside $\Her(\chi)$.}
Choose a maximal family of mutually disjoint sets
$c_{s,l}F\subseteq E_s$ with
$c_{s,l}\in\intt_{F^2}(E_s)$, and denote its cardinality by $R_s$.
If $cF\cap c'F\ne\varnothing$, then $c'\in cFF^{-1}$.
Hence each chosen tile excludes at most $|FF^{-1}|$ centers.
Maximality and the preceding interior estimate give
\begin{equation}\label{eq:local-return-rank}
 R_s\geq\frac{N_s}{2|FF^{-1}|}=\frac{N_s}{C_F}.
\end{equation}
Index these tiles by $0\leq l<R_s$, and take $c_{s,0}$ as the reference
center.  Put
\[
 p_s=e^s_{c_{s,0},c_{s,0}}.
\]
This is an open projection in $A^{**}$; it is not assumed to be in the
generated algebra.  The following are bounded multiplier columns of
the \emph{original} algebra $C_s$:
\begin{equation}\label{eq:local-columns}
 v_{s,l}(y)=\sum_{\delta\in F}
 \frac{\chi(yc_{s,l}\delta)^{1/2}}
      {S(yc_{s,l})^{1/2}}\,e^s_{c_{s,l}\delta,c_{s,0}}
 \qquad(y\in U_s).
\end{equation}
The denominator is at least $\kappa^{1/2}$.  Disjointness of the tiles
and the normalization imply
\begin{equation}\label{eq:local-column-relations}
 v_{s,l}^*v_{s,k}=\delta_{lk}p_s,\qquad
 q_{s,l}=v_{s,l}v_{s,l}^*.
\end{equation}
Every $q_{s,l}$ has its range in the support projection of $\chi$.
Indeed, its component at a level where $\chi$ vanishes is zero.
Consequently
\[
 T_sv_{s,l}v_{s,k}^*\in A\cap\Her(\chi).
\]
The factor $t_s\in C_0(U_s)$ guarantees membership in $A$.
The operator $T_s$ is central in $C_s$ and commutes with all its
multiplier columns and scalar coefficient functions.

Choose, simultaneously for all pairs $(s,l)$, labels
$\lambda_{s,l}\in[2,2.1]$ as in Lemma~\ref{lem:labels}, and set
\begin{equation}\label{eq:local-H}
 H=\sum_s T_s\sum_{l=0}^{R_s-1}\lambda_{s,l}q_{s,l}
       \in\Her(\chi)_+.
\end{equation}
Apply Lemma~\ref{lem:source-support} with $b=2.1$ to the functions
$e_\gamma$ in \eqref{eq:local-transported-remainders}. Their open
supports lie in that of $\alpha_\gamma^{-1}(r)$, and the return
centers satisfy the required $F^2$ buffer. We obtain the exact identity
\begin{equation}\label{eq:local-source-spectral-bound}
 (H-2.1\eta)_+e_\gamma^{1/2}=0,\qquad
 p_{e_\gamma}\leq1_{[0,2.1\eta]}(H)\quad(\gamma\in F).
\end{equation}
This is the support statement needed for the comparison links below.

\medskip\noindent
\emph{(2) Comparison and the reserved columns.}
Let
\[
 k_s=\left\lfloor\frac{\rho R_s}{Mm}\right\rfloor.
\]
By increasing the minimum tower height, assume
$k_s\geq\rho R_s/(2Mm)$.
Reserve, for each $\gamma\in F$, $m$ mutually disjoint sets
$J_{s,\gamma,1},\ldots,J_{s,\gamma,m}$ of $k_s$ column indices.
All these sets, for all $\gamma$, are disjoint.  Their union has at
most $\rho R_s$ elements.  Write
\[
 T_{s,\gamma}=\bigcup_{j=1}^mJ_{s,\gamma,j}.
\]
Avoid the reference column and, in addition, leave $M$ different
nonreserved column indices $l_{s,\gamma}$ for the frame links below.
This is possible for sufficiently large $R_s$.

Choose an open set $\Omega_s$ containing $B_s$, with compact closure in
the open plateau where $t_s=1$, and a positive function
$\beta_s\in C_0(U_s)$ with open support $\Omega_s$.
For a fixed $\gamma$, select one original center level from each of
$k_s$ distinct return tiles, and let $b_\gamma\in C(X)_+$ be the sum
of the copies of $\beta_s$ on these levels, over all $s$.
Its open support $O_\gamma$ is a \emph{Cartan} open set.  By invariance
and \eqref{eq:local-return-rank},
\begin{equation}\label{eq:local-cartan-target-size}
 \mu(O_\gamma)\geq\sum_s k_s\mu(B_s)
 \geq\frac{\rho}{2C_FMm}\sum_s N_s\mu(B_s)
 >\frac{\rho(1-\varepsilon)}{2C_FMm}.
\end{equation}
Equations~\eqref{eq:local-error-choice} and
\eqref{eq:local-small-remainder} give the strict COS inequality.
Hence
\[
 e_\gamma\precsim b_\gamma^{\oplus m}.
\]
For each of these $m$ copies, match its selected original center rows
with the column indices in $J_{s,\gamma,j}$.  If $a$ is a selected
center row and $l$ its matched column, the multiplier
$v_{s,l}e^s_{c_{s,0},a}$ carries that row to $q_{s,l}$.
Multiplication by $\beta_s^{1/2}$ makes this an element of $A$.
The relations in \eqref{eq:local-column-relations} show that these
matches give Cuntz equivalence of the matrix direct sum with
\begin{equation}\label{eq:local-folded-target}
 B_\gamma=\sum_s\sum_{l\in T_{s,\gamma}}\beta_s q_{s,l}
       \in\Her(\chi)_+.
\end{equation}
Thus $e_\gamma\precsim B_\gamma$ in $A$.
Lemma~\ref{lem:exact} gives $w_\gamma\in\Her(\chi)$ with
\begin{equation}\label{eq:local-comparison-links}
 w_\gamma^*w_\gamma=e_\gamma,\qquad
 w_\gamma w_\gamma^*\in\Her(B_\gamma).
\end{equation}
If $e_\gamma=0$, take $w_\gamma=0$ and omit its unused data.
The range of $w_\gamma$ is supported on its own reserved indices
$T_{s,\gamma}$ and on $t_s=1$.  On each of those rows the left
$H$-value is the constant $\lambda_{s,l}\in[2,2.1]$, whereas the
right $H$-spectrum is contained in $[0,2.1\eta]$ by
\eqref{eq:local-source-spectral-bound}.

\medskip\noindent
\emph{(3) The coefficients to be recovered.}
In every column algebra include the adjacent chain
$T_sv_{s,l}v_{s,l+1}^*$, $0\leq l<R_s-1$.
For every original level $a\in E_s$, the scalar base function
\begin{equation}\label{eq:local-base-data}
 F_{s,a}(y)=g_s(y)f(ya)
\end{equation}
satisfies $|F_{s,a}|\leq\eta^{-1}\|f\|t_s$.
Assign an unused off-diagonal position $(i,j)$ to the actual entry
\[
 F_{s,a}v_{s,i}v_{s,j}^*
   =T_s(F_{s,a}/t_s)v_{s,i}v_{s,j}^*.
\]
The quotient is continuous and bounded on $U_s$, and the entry belongs
to $A$ since $F_{s,a}\in C_0(U_s)$. Also assign positions to the $M$
entries $T_sc'_\gamma v_{s,i}v_{s,j}^*$, using the scalar functions
defined below in \eqref{eq:local-frame-defect}.

The other high/high data are the corners of
\begin{equation}\label{eq:local-compressed-data}
 A_\gamma=w_\gamma c_\gamma^{(0)}w_\gamma^*.
\end{equation}
For each pair of its reserved range rows $(s,p),(t,q)$, move the
corner to an unused off-diagonal pair of coding rows.  More explicitly,
if
\[
 c=v_{s,p}^*A_\gamma v_{t,q},
\]
place $v_{s,i}c v_{t,j}^*$ in the chosen position $((s,i),(t,j))$.
The coefficient $c$ is supported on $t_s=t_t=1$. In fact
\[
 c=(T_sv_{s,p})^*A_\gamma(T_tv_{t,q})\in A.
\]
Its translated corner is also in $A$, because
\[
 v_{s,i}cv_{t,j}^*=(T_sv_{s,i})c(T_tv_{t,j})^*.
\]
All four multiplied columns are actual elements of their tower
algebras.  For $s=t$ use a distinct unused unordered pair $\{i,j\}$,
with one orientation; even a diagonal corner is moved to such an
off-diagonal pair. Encode only original pairs $p\leq q$ when $s=t$:
restore an off-diagonal source corner together with its adjoint, and
restore a diagonal source corner once.
For $s<t$ use distinct pairs between the two
blocks and let adjoints encode the reverse corners.  The operators
$A_\gamma$ are self-adjoint, since $f$ is real-valued.

Here is explicit room for these assignments.  Put
$d_{s,\gamma}=|T_{s,\gamma}|=mk_s\leq\rho R_s/M$.
Within one block, it suffices to assign at most
\[
 (R_s-1)+N_s+M+\sum_{\gamma\in F}d_{s,\gamma}^2
 \leq (R_s-1)+C_FR_s+M+\frac{\rho^2}{M}R_s^2
\]
unordered pairs, each followed by its adjoint in $K$.
This is less than $R_s(R_s-1)/2$ with the stated choice of $R_*$.
Indeed, $R_s\geq16(C_F+M+2)\geq80$ bounds the linear terms by
$R_s^2/16$, while $\rho^2/M<10^{-4}$. The required number is
therefore less than $0.063R_s^2$, and the available number is
greater than $0.49R_s^2$.
Between two blocks the number of positions is at most
\[
 \sum_{\gamma\in F}d_{s,\gamma}d_{t,\gamma}
 \leq\frac{\rho^2}{M}R_sR_t<R_sR_t.
\]
These inequalities explain why one may require all heights to be
sufficiently large before choosing the towers.  They also avoid
counting a diagonal corner as a separate self-adjoint generator.

\medskip\noindent
\emph{(4) The frame links and the second self-adjoint element.}
Fix $s$ temporarily and omit its subscript.  In the reference tile set
\begin{equation}\label{eq:local-frame-defect}
 w'_{\gamma}=V_\gamma p,\qquad
 c'_{\gamma}(y)=\frac{\chi(yc_0\gamma)}{S(yc_0)^{1/2}},\qquad
 v_0^*w'_{\gamma}=c'_\gamma p,\qquad
 z_\gamma=w'_{\gamma}-v_0c'_{\gamma}.
\end{equation}
Here $c'_\gamma$ is the bounded scalar base function for which a position
was reserved above. We regard it as a central multiplier of the original
tower block; $c'_\gamma p$ is its corner on the reference source row.
The tower relations give
\[
 v_0^*z_\gamma=0,
 \qquad q_lz_\gamma=0\quad\text{for every }l,
 \qquad Hz_\gamma=0.
\]
For $l\ne0$, the second identity uses disjointness of the reference
tile and the other tiles.  Introduce the additional links
\begin{equation}\label{eq:local-extra-links}
 x_{s,\gamma}=T_sz_\gamma v_{s,l_{s,\gamma}}^*.
\end{equation}
They belong to $A\cap\Her(\chi)$.  Indeed, they are bounded
multiplier coefficients of the original tower block multiplied by
$t_s\in C_0(U_s)$, and both their endpoints lie in the support of
$\chi$.  They have left $H$-value $0$ and right joint values
$\lambda_{s,l_{s,\gamma}}t_s$.

Let $\mathcal P$ be the finite list of all the high/high entries just
assigned, including the chain, the scalar data, and the moved corners.
Define the self-adjoint element
\begin{equation}\label{eq:local-K}
 K=\sum_{x\in\mathcal P}(x+x^*)
     +\sum_{s,\gamma\in F}(x_{s,\gamma}+x_{s,\gamma}^*)
     +\sum_{\gamma\in F}(w_\gamma+w_\gamma^*)
       \in\Her(\chi).
\end{equation}
For the rest of the proof put $D=C^*(D_0,H,K)$.

\medskip\noindent
\emph{(5) Recovery of the columns and the remainder.}
Table~\ref{tab:local-types} records all terms of $K$.
Here $R_{ij}=\{(\lambda_i\tau,\lambda_j\tau):0\leq\tau\leq1\}$,
with the tower index included in a label index.

\begin{table}[htbp]
\centering\small\renewcommand{\arraystretch}{1.12}
\caption{Spectral types in the local construction. The adjoint of
each entry has the reversed type.}\label{tab:local-types}
\begin{tabularx}{\textwidth}{@{}P{.24\textwidth}P{.40\textwidth}Y@{}}
\toprule
Entry & Joint type for $H$ & Recovery \\
\midrule
Chain and scalar data
& $R_{ij}$, $i\ne j$, coefficient $O(t_s)$
& Smooth ray filter; origin cutoff \\
Moved corner of $A_\gamma$
& $\{(\lambda_i,\lambda_j)\}$, $i\ne j$
& Smooth point filter \\
Component of $w_\gamma$
& $\{\lambda_i\}\times[0,2.1\eta]$
& Reserved column and high cutoff \\
Frame defect $x_{s,\gamma}$
& $\{0\}\times[0,\lambda_{s,l_{s,\gamma}}]$
& Right column cutoff after subtraction \\
\bottomrule
\end{tabularx}
\end{table}

Every scalar high/high entry has joint spectral support on its
assigned ray
\[
 \{(\lambda_{s,i}t,\lambda_{s,j}t):0\leq t\leq1\}.
\]
Its coefficient is $O(t)$, including the scalar functions
$F_{s,a}$ and $t_sc'_\gamma$.  A moved remainder corner in one block
has its support at the endpoint of its assigned ray; one between
blocks has the single joint point
$(\lambda_{s,i},\lambda_{t,j})$.  The generic ratios make all these
assigned types distinct, including their adjoints.  By contrast,
the additional links have joint types on the axes, and the comparison
links have high/low or low/high joint types.  A comparison link has a left value in $[2,2.1]$ and a right
value in $[0,2.1\eta]$, or the reverse values for its adjoint. Such a
type cannot meet a high/high ray: a point on that ray has the ratio of
its coordinates in $[2/2.1,2.1/2]$. The axis types meet the high/high
rays only at the origin. This proves the required separation.

Lemma~\ref{lem:filter}, using smooth functions which avoid a shrinking
neighborhood of the origin, therefore recovers every member of
$\mathcal P$.  The omitted origin part is $O(t)$ and tends to zero in
norm.  Lemma~\ref{lem:soft} then recovers
\begin{equation}\label{eq:local-soft-units-known}
 E^s_{ij}=T_sv_{s,i}v_{s,j}^*,\qquad T_sq_{s,i}\quad(0\leq i,j<R_s).
\end{equation}
For example, a scalar entry $T_s\phi v_{s,i}v_{s,j}^*$ is transported
to the reference diagonal by
\[
 E^s_{0i}\,k_\delta(T_sq_{s,i})
    (T_s\phi v_{s,i}v_{s,j}^*)
       k_\delta(T_sq_{s,j})\,E^s_{j0}
       \longrightarrow T_s\phi q_{s,0},
 \qquad k_\delta(t)=\theta_\delta(t)/t.
\]
The product equals $\theta_\delta(t_s)^2T_s\phi q_{s,0}$; the
convergence follows from its $O(t_s)$ coefficient.
Thus $F_{s,a}q_{s,0}$ and $T_sc'_\gamma q_{s,0}$ are in $D$.
Similarly, the moved remainder corners are transported back to their
reserved rows using \eqref{eq:local-soft-units-known}.  Here no
inverse limit is needed, because these corners have $t_s=t_t=1$.
Summing the restored corners, using each diagonal corner once and
adding adjoints only for the off-diagonal source corners, gives
$A_\gamma\in D$.

Subtract the recovered high/high terms from $K$.
Choose a continuous function $G$ equal to $0$ on $[0,2.1\eta]$ and
equal to $1$ on $[2,2.1]$.
For the reserved range of a fixed $\gamma$, set
\[
 P_{\gamma,\delta}
       =\sum_s\sum_{l\in T_{s,\gamma}}
                      \theta_\delta(T_sq_{s,l})\in D.
\]
For $\delta<1/2$, this acts as the identity on the range of
$w_\gamma$, since that range has $t_s=1$.  Other comparison links
have disjoint reserved ranges.  The adjoints of all comparison links
have low left $H$-spectrum.  The $x_{s,\gamma}$ have left value $0$,
and their adjoints have high ranges on the indices
$l_{s,\gamma}$, all outside the reserved sets.  It follows that
\[
 P_{\gamma,\delta}G(H)
       \left(K-\sum_{x\in\mathcal P}(x+x^*)\right)=w_\gamma.
\]
Hence \emph{each} $w_\gamma\in D$.
The relations in \eqref{eq:local-comparison-links} and
\eqref{eq:local-compressed-data} give
\[
 e_\gamma\in D,
 \qquad w_\gamma^*A_\gamma w_\gamma
             =e_\gamma c_\gamma^{(0)}e_\gamma\in D.
\]
Division by soft cutoffs of $e_\gamma$, as in
Lemma~\ref{lem:division}, recovers
\begin{equation}\label{eq:local-remainder-recovered}
 c_\gamma^{(0)}\in D.
\end{equation}

After subtracting all comparison links, the remaining expression is
$\sum_{s,\gamma}(x_{s,\gamma}+x_{s,\gamma}^*)$.
Right multiplication by
$\theta_\delta(T_sq_{s,l_{s,\gamma}})$ selects
$x_{s,\gamma}\theta_\delta(t_s)$: the other links have different
right column indices, and the adjoints have right supports orthogonal
to every column.  This converges in norm to $x_{s,\gamma}$, because
$x_{s,\gamma}$ has an $O(t_s)$ coefficient.  Therefore all these
links belong to $D$.

\medskip\noindent
\emph{(6) Recovery of the diagonal function.}
Again fix one block and suppress $s$.  From its recovered link and
soft matrix units,
\[
 x_\gamma k_\delta(Tq_{l_\gamma})E_{l_\gamma,0}
       =\theta_\delta(t)Tz_\gamma v_0^*
       \longrightarrow Tz_\gamma v_0^*.
\]
Adding the recovered scalar term $Tc'_\gamma q_0$ gives
\begin{equation}\label{eq:local-Bgamma}
 L_\gamma=T w'_\gamma v_0^*\in D.
\end{equation}
Notice that $p$ has not been recovered. Instead, the recovered
$L_\gamma$ records its image under each known frame operator.

Let $\phi=F_{s,a}$ be one of the scalar functions in
\eqref{eq:local-base-data}. Define the actual diagonal function
$A_\phi=\phi p\in C(X)$ on the reference level $U_sc_{s,0}$, extended
by zero. It belongs to $C(X)$ because $\phi\in C_0(U_s)$.
The corresponding column corner $Y_\phi=\phi q_0$ has already been
recovered. Put
\[
 b_\delta(t)=\theta_\delta(t)/t^2,
\]
extended by zero near zero. Functional calculus on the known element
$Tq_0$, followed by the column relations, gives
\begin{equation}\label{eq:local-frame-corners}
 L_\gamma b_\delta(Tq_0)Y_\phi L_\nu^*
   =\theta_\delta(t)V_\gamma A_\phi V_\nu^*
       \longrightarrow V_\gamma A_\phi V_\nu^*.
\end{equation}
Here $\theta_\delta(t)$ is the scalar base cutoff on the indicated
tower entries. The frame buffer ensures that all their endpoints
remain in the same original block. Since
$|\phi|\leq\eta^{-1}\|f\|t$, the error in
\eqref{eq:local-frame-corners} is at most
$2\delta\eta^{-1}\|f\|\|V_\gamma\|\|V_\nu\|$.
Thus each $V_\gamma A_\phi V_\nu^*$ belongs to $D$.

Apply Lemma~\ref{lem:frame} to the known frame
$(V_\gamma)_{\gamma\in F}$ and the just recovered inner corners.
It gives
\begin{equation}\label{eq:local-frame-reconstruction}
 A_\phi=S^{-1}\left(\sum_{\gamma,\nu\in F}
       V_\gamma^*(V_\gamma A_\phi V_\nu^*)V_\nu\right)S^{-1}\in D.
\end{equation}
The middle sum equals $SA_\phi S$. Every outer operator in this
identity belongs to $D_0$. In particular, no inverse of a compressed
frame and no membership assertion for $p$ is needed.

For the chosen $\phi=F_{s,a}$, the diagonal function on the original
level $U_sa$ is
\[
 u_{a^{-1}c_{s,0}}A_\phi u_{a^{-1}c_{s,0}}^*.
\]
Its value at $ya$ is $A_\phi(yc_{s,0})=g_s(y)f(ya)$, by the
right-action convention. Its support is contained in $U_sa$.
All these unitaries belong to $D_0$, so summing over the levels gives
$hf\in D$.
Finally, \eqref{eq:local-remainder-recovered}, the group unitaries,
and \eqref{eq:local-exact-diagonal-decomposition} give $f\in D$.
\end{proof}

\begin{remark}\label{rem:local-known-unitaries}
The hypothesis that all implementing group unitaries are already
available cannot simply be removed from Proposition~\ref{prop:local}.
The finite frame depends on $\chi$, and the final translates of its
recovered base functions use the full group.  Neither operation is
inferred from weak tracial approximation alone.
\end{remark}

\section{Countable assembly and the proof of the main theorem}

We first place local constructions in orthogonal diagonal hereditary
subalgebras. This separation is spectral, so both elements of each
local pair can subsequently be extracted by the same cutoff.

\begin{lemma}\label{lem:bands}
Every nonempty open subset $W$ of $X$ contains the compact support of a
positive function $z$ for which there are pairwise disjoint open
intervals $I_n\subset(0,\infty)$ meeting its range. There are nonzero
$\chi_n=\psi_n(z)$ and continuous functions $\vartheta_n$ supported in
$I_n$ such that
\[
 \supp\psi_n\Subset I_n,\qquad
 \vartheta_n\psi_n=\psi_n,\qquad
 \vartheta_n\psi_m=0\quad(m\ne n).
\]
If $a_m\in\Her(\chi_m)$, then
$\vartheta_n(z)a_m=a_m\vartheta_n(z)=\delta_{nm}a_m$.
\end{lemma}

\begin{proof}
The space has no isolated points: otherwise the orbit of an isolated
point would have invariant closed complement, and minimality would
make $X$ a compact discrete orbit, contradicting freeness of the
infinite group. Choose a smaller open set with closure in $W$, a
sequence of distinct points $x_n$ converging to a point there, and
disjoint small neighborhoods of the $x_n$. Choose bump functions
$b_n$ of compact support in those neighborhoods, with $b_n(x_n)=1$.
Then $z=\sum_n4^{-n}b_n$ converges uniformly, has compact support in
$W$, and takes each value $4^{-n}$. Take disjoint intervals around
these values and choose the indicated functions. The last assertion
holds first on $\chi_mA\chi_m$ and then on its norm closure.
\end{proof}

\begin{proof}[Proof of Theorem~\ref{thm:main}]
If $\Gamma$ is finite, each orbit is closed. Minimality and freeness
therefore imply that $X$ is one orbit with $|X|=|\Gamma|$, and
$A\cong M_{|\Gamma|}(\C)$. A diagonal self-adjoint matrix with distinct
positive entries, together with the self-adjoint sum of the adjacent
matrix units and their adjoints, generates this matrix algebra.

Suppose henceforth that $\Gamma$ is infinite. Let
$s_1,\ldots,s_k$ generate $\Gamma$. Apply
Proposition~\ref{prop:finite} to the real and imaginary parts of the
$u_{s_j}$. Obtain a pair $H_0,K_0$, three unused positions, and an open
plateau $W$. Then $D_0=C^*(1,H_0,K_0)$ contains every $u_\gamma$.

Choose $z,\chi_n,\vartheta_n$ as in Lemma~\ref{lem:bands}, and let
$(f_n)$ be norm dense in $C(X)_\sa$. By
Proposition~\ref{prop:local}, for each $n$ there are
$H_n\in\Her(\chi_n)_+$ and $K_n\in\Her(\chi_n)_\sa$ such that
\[
 f_n\in C^*(1,\{u_\gamma:\gamma\in\Gamma\},\chi_n,H_n,K_n).
\]
Replace $H_n$ and $K_n$, respectively, by
$H_n/(1+\|H_n\|)$ and $K_n/(1+\|K_n\|)$. Both norms are then at most
one, and multiplication by the original scalar constants recovers the
unscaled pair. The series
\[
 P=\sum_{n=1}^\infty2^{-n}H_n,
 \qquad Q=\sum_{n=1}^\infty2^{-n}K_n
\]
converge in norm and belong to $\Her(W)$. Encode the three
self-adjoint elements $z,P,Q$ in the unused positions of
Proposition~\ref{prop:finite}. For the resulting $K=K^*$,
$D=C^*(1,H_0,K)$ contains $D_0,z,P,Q$. Moreover,
\[
 2^n\vartheta_n(z)P\vartheta_n(z)=H_n,
 \qquad
 2^n\vartheta_n(z)Q\vartheta_n(z)=K_n.
\]
Thus $D$ contains every $f_n$, and therefore $C(X)$ and all
implementing unitaries. Hence $D=A$.

Set $z_0=1+H_0+iK$. Since
$\operatorname{Re}z_0=1+H_0$ has spectrum in $[1,1+\|H_0\|]$, choose
a continuous scalar function vanishing at zero and equal to one on
this interval. Its functional calculus shows that $1\in C^*(z_0)$.
Consequently $H_0,K\in C^*(z_0)$, and $C^*(z_0)=A$.
\end{proof}

\begin{remark}
Finite generation is used only in the first paragraph of the
infinite-group argument: recovery of the finitely many $u_{s_j}$ gives
every $u_\gamma$. The finite Fourier proposition itself is valid for
countably infinite amenable groups. The local proposition is also a
countable-group statement, but it is relative to the implementing
unitaries. Since the finite frame depends on the prescribed local
support, these two propositions do not by themselves establish the
theorem for arbitrary countable groups. A countable-group extension
would require a construction which recovers the unitaries needed by
each successive local frame while preserving all previously recovered
data. No such extension is claimed here.
\end{remark}

\begin{remark}
The proof uses the diagonal functions and the explicit matrix
coordinates of the Rokhlin towers. Weak tracial approximation alone
does not provide the finite-frame recovery formula
\eqref{eq:local-frame-reconstruction}. The result proved here is
therefore a theorem for the stated dynamical hypotheses, rather than
a statement about all weakly tracially homogeneous $C^*$-algebras.
\end{remark}

\end{document}